\documentclass[11pt,A4paper]{article}

\usepackage[left=32mm,right=32mm,top=30mm,bottom=32mm]{geometry}
\usepackage{labelfig}
\usepackage{epsfig}
\usepackage{epstopdf}

\usepackage{xfrac}

\usepackage[percent]{overpic}

\usepackage{amsthm}

\usepackage{color}
\usepackage{amsthm,amsmath,amssymb}
\usepackage{booktabs}
\usepackage{mathpazo}
\usepackage{microtype}
\usepackage{overpic}
\usepackage{bm}
\usepackage{sectsty}
\usepackage[
	pdftitle={PDFTitle},
	pdfauthor={Hugo Parlier},
	ocgcolorlinks,
	linkcolor=linkred,
	citecolor=linkred,
	urlcolor=linkblue]
{hyperref}

\definecolor{linkred}{RGB}{128,0,128}
\definecolor{linkblue}{RGB}{16, 78, 139}

\usepackage[hang,flushmargin]{footmisc}
\usepackage{enumitem}
\usepackage{titlesec}
	\titlespacing{\section}{0pt}{12pt}{0pt}
	\titlespacing{\subsection}{0pt}{6pt}{0pt}
	
\titlelabel{\thetitle.\quad}

\makeatletter 

\long\def\@footnotetext#1{%
\H@@footnotetext{%
\ifHy@nesting 
\hyper@@anchor{\@currentHref}{#1}%
\else 
\Hy@raisedlink{\hyper@@anchor{\@currentHref}{\relax}}#1%
\fi 
}}

\def\@footnotemark{%
\leavevmode 
\ifhmode\edef\@x@sf{\the\spacefactor}\nobreak\fi 
\H@refstepcounter{Hfootnote}%
\hyper@makecurrent{Hfootnote}%
\hyper@linkstart{link}{\@currentHref}%
\@makefnmark 
\hyper@linkend 
\ifhmode\spacefactor\@x@sf\fi 
\relax 
}%

\ifFN@multiplefootnote%
\renewcommand*\@footnotemark{%
\leavevmode 
\ifhmode 
\edef\@x@sf{\the\spacefactor}%
\FN@mf@check 
\nobreak 
\fi 
\H@refstepcounter{Hfootnote}%
\hyper@makecurrent{Hfootnote}%
\hyper@linkstart{link}{\@currentHref}%
\@makefnmark 
\hyper@linkend 
\ifFN@pp@towrite 
\FN@pp@writetemp 
\FN@pp@towritefalse 
\fi 
\FN@mf@prepare 
\ifhmode\spacefactor\@x@sf\fi 
\relax%
}%
\fi 

\makeatother 

\theoremstyle{plain}
\newtheorem{theorem}{Theorem}[section]

\newtheorem{lemma}[theorem]{Lemma}
\newtheorem{corollary}[theorem]{Corollary}

\makeatletter
\newtheorem*{rep@theorem}{\rep@title}
\newcommand{\newreptheorem}[2]{%
\newenvironment{rep#1}[1]{%
 \def\rep@title{#2 \ref{##1}}%
 \begin{rep@theorem}}%
 {\end{rep@theorem}}}
\makeatother

\newreptheorem{theorem}{Theorem}
\newreptheorem{corollary}{Corollary}

\theoremstyle{definition}

\newtheorem{remark}[theorem]{Remark}

\newcommand{\Hyp}{{\mathbb H}}

\newcommand{\M}{{\mathcal M}}

\newcommand{\area}{{\rm area}}
\newcommand{\arcsinh}{{\,\rm arcsinh}}
\newcommand{\arccosh}{{\,\rm arccosh}}

\newcommand{\injrad}{{\rm injrad}}

\sectionfont{\large \bfseries}
\subsectionfont{\normalsize}

\long\def\symbolfootnote[#1]#2{\begingroup%
\def\thefootnote{\fnsymbol{footnote}}\footnote[#1]{#2}\endgroup}

\def\blfootnote{\xdef\@thefnmark{}\@footnotetext}

\begin{document}

{\Large \bfseries Geometric filling curves on surfaces}

{\large Ara Basmajian\symbolfootnote[1]{\normalsize Research supported by a PSC-CUNY Grant and a Simons foundation grant}, Hugo Parlier\symbolfootnote[7]{\normalsize Research supported by Swiss National Science Foundation grant number PP00P2\textunderscore 153024 }
 and Juan Souto\symbolfootnote[12]{\normalsize Research supported by the National Science Foundation Grant No. DMS-1440140 while the third author was in residence at the Mathematical Sciences Research Institute in Berkeley, California, during the Fall 2016 semester
}}

{\bf Abstract.} 
This note is about a type of quantitative density of closed geodesics on closed hyperbolic surfaces. The main results are upper bounds on the length of the shortest closed geodesic that $\varepsilon$-fills the surface. 

\vspace{1cm}

\section{Introduction}

Closed geodesics on hyperbolic surfaces provide a concrete link between algebraic, geometric and topological approaches to understanding the geometry of these surfaces and their moduli spaces. One feature of the hyperbolic geometry is that they are dense in both the surface and the unit tangent bundle. This is in strong contrast to the set of simple closed geodesics or those with bounded intersection number which are nowhere dense and in fact Haussdorf dimension $1$ \cite{Birman-Series}. In this article, we investigate a type of quantitative density of closed geodesics.

Given a closed hyperbolic surface $X$ and $\varepsilon>0$, we're interested in finding the shortest closed geodesic that is $\varepsilon$-dense, by which we mean that all points of $X$ are at distance at most $\varepsilon$ from the geodesic. Our main result is the following. 
\begin{theorem} \label{thm:denseX}
For all $X\in \M_g$ there exists a constant $C_X>0$ and such that for all $\varepsilon \leq \frac{1}{2}$ there exists a closed geodesic $\gamma_\varepsilon$ that is $\varepsilon$-dense on $X$ and such that
$$
\ell(\gamma_\varepsilon) \leq C_X \frac{1}{\varepsilon}  \log\left(\frac{1}{\varepsilon}\right)
$$
\end{theorem}

Another measure of complexity for a closed geodesic is its self-intersection number. Length and  self-intersection numbers of a curve are of course related (see for instance \cite{BasmajianUniversal}). Instead of minimizing length, one can try and minimize self-intersection for $\varepsilon$ filling curves. As a corollary of Theorem \ref{thm:denseX} we obtain the following.
\begin{corollary}\label{cor:int}
For all $X\in \M_g$ there exists a constant $C_X>0$ such that for all $\varepsilon \leq \frac{1}{2}$ there exists a closed geodesic $\gamma_\varepsilon$ that is $\varepsilon$-dense on $X$ and such that
$$
i(\gamma_\varepsilon, \gamma_\varepsilon) \leq C_X \frac{1}{\varepsilon^2}  \left(\log\left(\frac{1}{\varepsilon}\right)\right)^2
$$
\end{corollary}

The main ingredient in the proof of Theorem \ref{thm:denseX} is a more technical result (Theorem \ref{thm:maintech}) which shows the existence of a closed geodesic of bounded length that contains a given set of geodesic segments on $X$ in its $\varepsilon$-neighborhood. Theorem \ref{thm:denseX} then follows by finding an appropriate set of geodesic segments that fill the surface. This tool can also be used to find a similar quantitive density result in the unit tangent bundle of $X$ (Theorem \ref{thm:tangent}).

The growth rate of length in Theorem \ref{thm:denseX} is perhaps not optimal but if not it is off by at most a factor of $ \log\left(\frac{1}{\varepsilon}\right)$. Indeed if a closed geodesic is $\varepsilon$-dense then the area of its $\varepsilon$-neighborhood is the area of the surface. In $\Hyp$, for small $\varepsilon$, the area of a $\varepsilon$-neighborhood of a geodesic segment of length $\ell$ is roughly $\varepsilon \ell$. And putting these things together tells us that if a geodesic $\gamma$ $\varepsilon$-fills, it must be of length at least $\frac{\area(X)}{\varepsilon}$. Understanding the $\log\left(\frac{1}{\varepsilon}\right)$ discrepancy between the upper and lower bounds seems to be an interesting problem.

{\bf Acknowledgement.}

The authors acknowledge support from U.S. National Science Foundation grants DMS 1107452, 1107263, 1107367 RNMS: Geometric structures And Representation varieties (the GEAR Network) 

\section{Geodesics in $\Hyp$ and on surfaces}

We begin with a simple lemma about hyperbolic polygons and angles of geodesics traversing them. By angle between two geodesics we mean the minimal angle, so in particular angles are always less than $\frac{\pi}{2}$. 

\begin{lemma}\label{lem:poly}
Let $P$ be a convex polygon in $\Hyp$. Then there exists a $\theta_P>0$ such that any geodesic segment with endpoints on distinct sides of $\partial P$ forms an angle of at least $\theta_P$ in one of its endpoints. 
\end{lemma}

\begin{proof}
The statement follows by a compactness argument since a geodesic segment cannot have angle $0$ in both endpoints, but there is a direct argument that also shows that this minimal angle corresponds to a specific geometric quantity which we now describe.

Consider the set of all geodesic segments that form a triangle with consecutive sides of $P$. We'll call these triangles the {\it ears} of $P$. Consider the set of angles of the ears (three angles for each ear) and set $\theta_P$ to be their minimum value. 

\begin{figure}[h]
{\color{linkred}
\leavevmode \SetLabels
\L(0.37*.71) $s$\\
\L(0.553*.81) $\tilde{s}$\\
\L(0.52*.62) $T$\\
\L(0.47*.81) $T'$\\
\endSetLabels
\begin{center}
\AffixLabels{\centerline{\epsfig{file =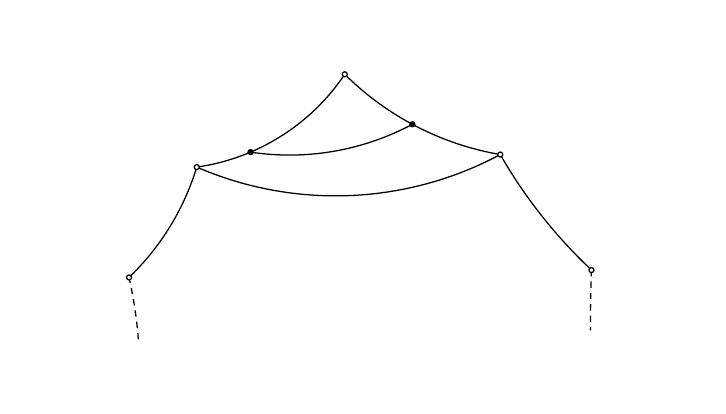,width=7.0cm,angle=0} }}
\vspace{-30pt}
\end{center}
\caption{The ear $T$ containing $T'$} \label{fig:polygonear}
}
\end{figure}
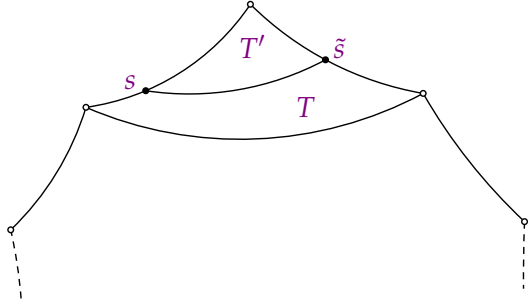

Now consider a geodesic segment $c$ leaving from a side $s$ of $\partial P$ forming an angle $\theta$ of at most $\theta_P$. It stays entirely in an ear $T$ (one of the sides of the ear is $s$) so it intersects a side $\tilde{s}$ of $P$ adjacent to $s$. The triangle $T'$ formed by $c$ with segments of $s$ and $\tilde{s}$ is contained in $T$ and is thus of lesser area. As it also shares an angle with $T$, the sum of its two remaining angles cannot be strictly less than $2 \theta_P$, otherwise by Gauss-Bonnet the area of $T'$ would be greater than that of $T$. As $\theta \leq \theta_P$, the remaining angle is at least $\theta_P$.

Note that $\theta_P$ is optimal as the inner sides of the ears are admissible geodesic segments. 
\end{proof}

The following is just an observation about geodesics in $\Hyp$. The proof we give, and many of the following proofs, use hyperbolic trigonometry. We refer the reader to \cite{BuserBook} or any other standard hyperbolic geometry textbook for the formulas.

\begin{lemma}\label{lem:angle}
Let $\frac{\pi}{2} \geq \theta_0>0$, and set
$$m(\theta_0):= \arccosh\left( \frac{2}{\sin^2\left(\theta_0\right)} - 1\right)$$
If $c$ is a geodesic segment in $\Hyp$ of length at least $m(\theta_0)$ between two (complete) geodesics $\gamma_1,\gamma_2$ such that 
$$
\angle(c,\gamma_i) \geq \theta_0
$$
for $i=1,2$, then $\gamma_1$ and $\gamma_2$ are disjoint. 
\end{lemma}

\begin{proof}
Fix $\theta_0$. To check whether segments of given length $\ell$ with the angle condition always lie give rise to disjoint geodesics, we can consider the "worst case scenario" which is when both geodesics form an angle of $\theta_0$ with the segment. The limit case in this worst case scenario between intersecting and not intersecting is when the two geodesics $\gamma_1$ and $\gamma_2$ are ultra-parallel and in this case we have a triangle with an ideal vertex and two angles of $\theta_0$ in endpoints of a segment $c$ of length $\ell$. In this case $\ell$ is easy to compute via standard hyperbolic trigonometry: it is 
$$
m(\theta_0)=\arccosh\left( \frac{2}{\sin^2\left(\theta_0\right)} - 1\right)
$$
The distance between $\gamma_1$ and $\gamma_2$ being an increasing function of $\ell$, we can conclude that any $c$ of length greater or equal to $m(\theta_0)$ with the angle condition will lie between disjoint geodesics. 
\end{proof}
Note that the proof also shows that the condition is sharp. The next lemma is more technical and contains several quantifiers, but is again a somewhat elementary statement about geodesics in $\Hyp$.

\begin{lemma}\label{lem:tech}
Let $\frac{\pi}{2} \geq \theta_0 >0$ be a fixed constant. Let $c$ be a geodesic segment in $\Hyp$ and $\gamma$ the complete geodesic containing $c$. Fix $\varepsilon>0$ and let $\gamma_1$ and $\gamma_2$ be geodesics that intersect $\gamma$ such that the intersection points $p_1$, $p_2$ lie on different sides of $c$.

Suppose for $i=1,2$ that $
\angle (\gamma_i,\gamma) \geq \theta_0$
and
$$
d(c,p_i) \geq \log\left(\frac{1}{\varepsilon}\right) +  \log\left(\frac{4}{\sin(\theta_0)}\right)
$$
Then any geodesic $\delta$ intersecting both $\gamma_1$ and $\gamma_2$ satisfies
$$
c \subset B_\varepsilon(\delta)
$$
\end{lemma}

\begin{proof}

Given $\theta_0$ and the length of $c$, we'll consider the "worst case scenario" by which we mean the situation, under the assumptions of the lemma, where a geodesic with endpoints on $\gamma_1$ and $\gamma_2$ is as far away as possible from $c$. 

In order for it to be the worst case scenario a number of things need to happen. First of all, given any setup for $c$ and $\gamma_1$ and $\gamma_2$ the geodesic furthest away from $c$ with endpoints on $\gamma_1$ and $\gamma_2$ is a limit case and actually is a geodesic $\delta$ ultra-parallel to both $\gamma_1$ and $\gamma_2$. Furthermore, decreasing the distance between $c$ and $p_i$ pushes geodesics away, so we can suppose that 
$$
d(c,p_i) = \log\left(\frac{1}{\varepsilon}\right) + \log\left(\frac{4}{\sin(\theta_0)}\right) =:r_\varepsilon
$$
Finally, if one of the angles is greater than $\theta_0$ then decreasing the angle also pushes $c$ away from the limit case so we can suppose that both angles of intersection are exactly $\theta_0$. We can thus suppose that we are in the situation of Figure \ref{fig:wcs}. 

\begin{figure}[h]
{\color{linkred}
\leavevmode \SetLabels
\L(0.267*.89) $\theta_0$\\
\L(0.70*.89) $\theta_0$\\
\L(0.484*.94) $c$\\
\L(0.62*.225) $\delta$\\
\L(0.475*.55) $d$\\
\L(0.24*.49) $\gamma_1$\\
\L(0.725*.49) $\gamma_2$\\
\L(0.6*.705) $\mu$\\
\L(0.22*.96) $p_1$\\
\L(0.746*.96) $p_2$\\
\endSetLabels
\begin{center}
\AffixLabels{\centerline{\epsfig{file =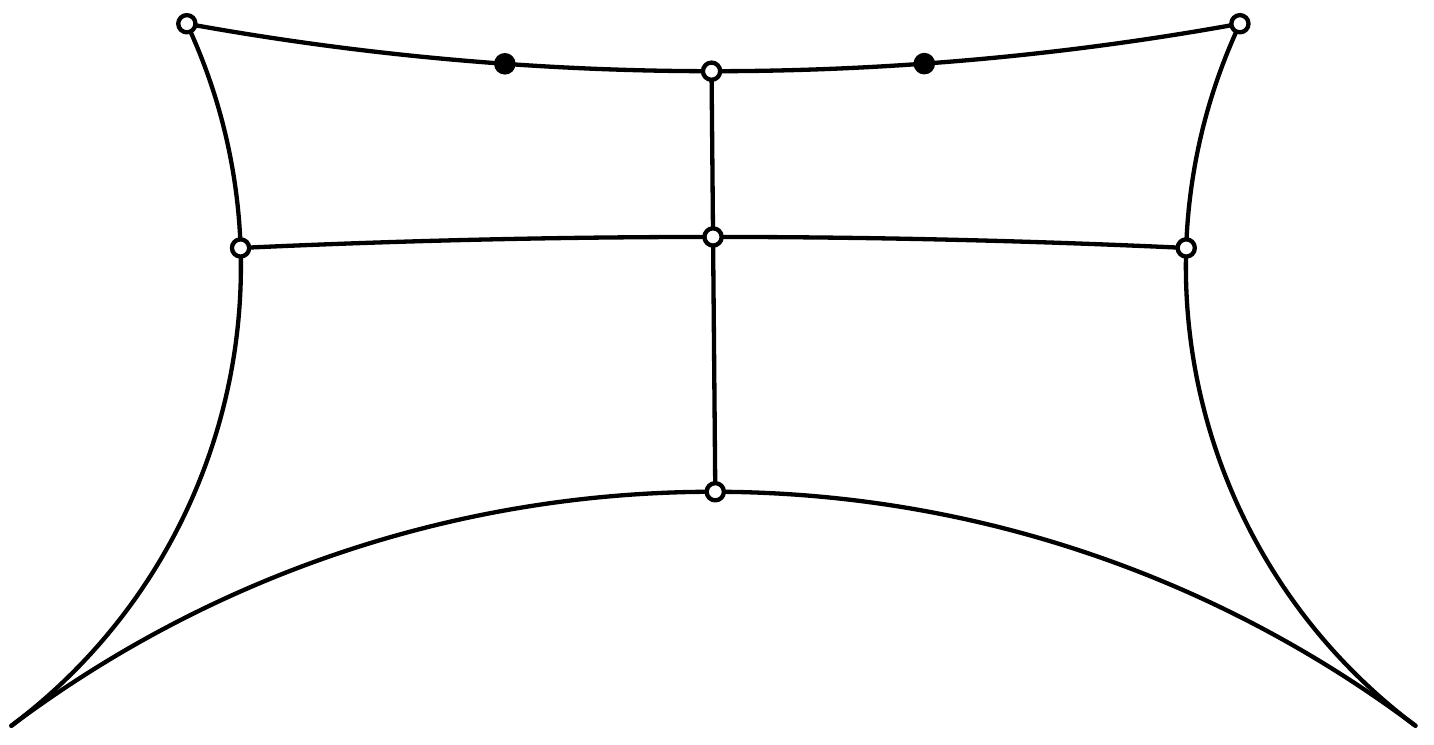,width=10.0cm,angle=0} }}
\vspace{-30pt}
\end{center}
\caption{The worst case scenario} \label{fig:wcs}
}
\end{figure}
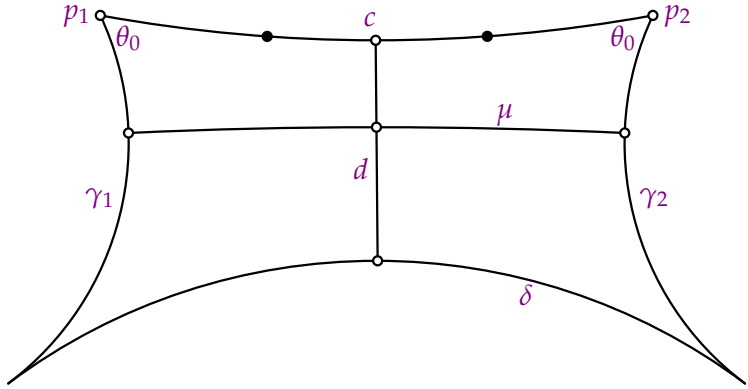

This situation has a lot of symmetry which we'll now use. Let $\mu$ be the minimal geodesic path between $\gamma_1$ and $\gamma_2$ and let $d$ be the the distance between $\gamma$ and $\delta$. We denote by $h$ the maximal distance between $c$ and $\delta$. It is this quantity that we need to bound in function of the other parameters. It lies in a symmetric quadrilateral $Q$ with $c$ as one of its sides

We denote by $d''$ the distance between $\gamma$ and $\mu$ and by $d'$ the distance between $\mu$ and $\delta$. By looking at the quadrilaterals separated by $\mu$, observe that $d'>d''$ and thus that $d' > \frac{d}{2}$. We consider another symmetric quadrilateral $Q'$, somewhat similar to $Q$ but this time with the height of length $d$ replaced with a height of length $d'$ (see Figure \ref{fig:doublequad}). 

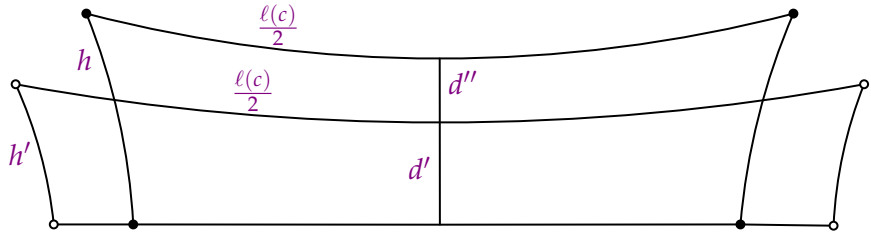
\begin{figure}[h]
{\color{linkred}
\leavevmode \SetLabels
\L(0.185*.70) $h$\\
\L(0.125*.3) $h'$\\
\L(0.475*.24) $d'$\\
\L(0.51*.60) $d''$\\
\L(0.342*.845) $\frac{\ell(c)}{2}$\\
\L(0.32*.57) $\frac{\ell(c)}{2}$\\
\endSetLabels
\begin{center}
\AffixLabels{\centerline{\epsfig{file =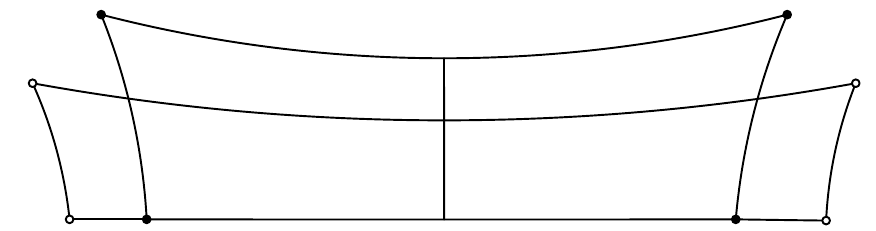,width=12.0cm,angle=0} }}
\vspace{-30pt}
\end{center}
\caption{The quadrilaterals $Q$ and $Q'$} \label{fig:doublequad}
}
\end{figure}

We denote by $h'$ the length of the side corresponding to $h$ and we claim that in fact $h'$ enjoys the property of $h' > \frac{h}{2}$. This will be useful because $h'$ is somewhat easier to compute. 

To prove this, using the left right symmetries of $Q$ and $Q'$, we first restrict ourselves to the left half of each quadrilateral (see Figure \ref{fig:doublequad}).

Using hyperbolic trigonometry, we have 
$$
\sinh(h')= \sinh(d') \cosh\left( \frac{\ell(c)}{2}\right) > \sinh\left( \frac{d}{2} \right)
$$
whereas
$$
\sinh(h)= \sinh(d) \cosh\left( \frac{\ell(c)}{2}\right) 
$$
Thus
$$
\frac{\sinh(h')}{\sinh(h)} > \frac{\sinh\left( \frac{d}{2} \right)}{\sinh(d)} = \frac{1}{2 \cosh\left( \frac{d}{2} \right)}
$$
But if $h' \leq \frac{h}{2}$ we would have 
$$
\frac{\sinh(h')}{\sinh(h)} \leq \frac{\sinh\left( \frac{h}{2} \right)}{\sinh(h)} = \frac{1}{2 \cosh\left( \frac{d}{2} \right)}
$$
a contradiction showing $h' > \frac{h}{2}$.

We now seek to bound $h'$ and to do so, we begin by computing $\ell(\mu)$. Using for instance the upper left quadrilateral of Figure \ref{fig:wcs} and hyperbolic trigonometry, we have
$$
\cosh\left(\frac{\ell{\mu}}{2}\right) = \cosh(r_\varepsilon + \sfrac{\ell(c)}{2}) \sin(\theta+0)
$$
Using the lower left quadrilateral we have
$$
\sinh(d') = \frac{1}{\sinh\left(\frac{\ell{\mu}}{2}\right)}
$$
We now return to the quadrilateral pictured in Figure \ref{fig:quadprime} to compute $h'$.

\begin{figure}[h]
{\color{linkred}
\leavevmode \SetLabels
\L(0.322*.52) $h'$\\
\L(0.67*.40) $d'$\\
\L(0.48*.83) $\frac{\ell(c)}{2}$\\
\endSetLabels
\begin{center}
\AffixLabels{\centerline{\epsfig{file =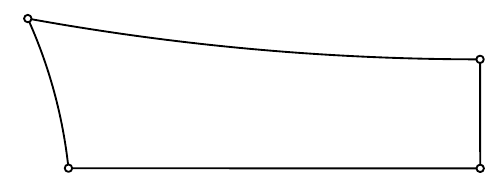,width=5.5cm,angle=0} }}
\vspace{-30pt}
\end{center}
\caption{The left half of $Q'$} \label{fig:quadprime}
}
\end{figure}

We have 
$$
\sin (\theta_0) \cosh(r_\varepsilon + \sfrac{\ell(c)}{2})  \sinh(h')= \cosh\left(\frac{\ell(c)}{2}\right)
$$
from which we deduce
$$
h' = \arcsinh\left( \frac{1}{\sin(\theta_0)} \frac{ \cosh\left(\frac{\ell(c)}{2}\right)}{\cosh\left(r_\varepsilon+ \frac{\ell(c)}{2}\right)}\right)
$$
We want to show that $h' \leq \frac{\varepsilon}{2}$ thus that 
\begin{equation}\label{eqn:main}
\frac{1}{\sin(\theta_0)} \frac{ \cosh\left(\frac{\ell(c)}{2}\right)}{\cosh\left(r_\varepsilon+ \frac{\ell(c)}{2}\right)} \leq  \sinh\left( \frac{\varepsilon}{2}\right)
\end{equation}
The left hand of Inequality \ref{eqn:main} can be manipulated to show that 
$$
\frac{1}{\sin(\theta_0)} \frac{ \cosh\left(\frac{\ell(c)}{2}\right)}{\cosh\left(r_\varepsilon+ \frac{\ell(c)}{2}\right)} = \frac{1}{\sin(\theta_0)} \frac{1 + e^{-\ell(c)}}{e^r+ e^{-r_\varepsilon-\ell(c)}}
$$
Note that 
$$
\frac{1 + e^{-\ell(c)}}{e^{r_\varepsilon}+ e^{-r_\varepsilon-\ell(c)}} < \frac{2}{e^{r_\varepsilon}}
$$
and thus Inequality \ref{eqn:main} certainly holds provided
$$
\frac{1}{\sin(\theta_0)} \frac{2}{e^{r_\varepsilon}} \leq \sinh\left( \frac{\varepsilon}{2}\right)
$$
and thus will hold if
$$
\frac{1}{\sin(\theta_0)} \frac{2}{e^{r_\varepsilon}} \leq  \frac{\varepsilon}{2}
$$
Expressed different this last inequality becomes
$$
r_\varepsilon \geq \log\left(\frac{1}{\varepsilon}\right) + \log\left(\frac{4}{\sin(\theta_0)}\right)
$$
As the above inequality is in fact an equality, this proves the lemma.
\end{proof}

We now can proceed to prove the main tool for our results.

\begin{theorem}\label{thm:maintech}
For any $X\in \M_g$, there exists a constant $K_X$ such that the following holds. For all  $ 1 > \varepsilon>0$ and any finite collection $\{c_i\}_{i=1}^N$ of geodesic segments of average length $\overline{c}$ on $X$, there exists a closed geodesic $\gamma$ such that
$$
\ell(\gamma) \leq N \left( K_X + \overline{c} + 2 \log\left(\frac{1}{\varepsilon}\right) \right)
$$
and for all $i=1,\hdots,N$
$$
c_i \subset B_\varepsilon(\gamma)
$$
\end{theorem}

\begin{proof}
We take a filling closed geodesic $\gamma_0$ on $X$ of minimal length (among filling geodesics). Thus $X\setminus \gamma_0$ consists in a finite collection of polygons $\{P_i\}_{i\in I}$ and we denote by $D$ the maximum their intrinsic diameters. Furthermore we denote by $\theta_0$ the minimum of $\{\theta_{P_i}\}_{i\in I}$ where the $\theta_{P_i}$s are from Lemma \ref{lem:poly}. 

As $X$ is orientable, the geodesic $\gamma_0$ has two sides which we think of as being $+$ and $-$ (it does not matter which is which but we fix it). We take $\mu_{+}$, resp. $\mu_{-}$, to be a geodesic arc from $\gamma$ to itself, orthogonal $\gamma$ in both end points, and which leaves and returns to the $+$ side, resp. $-$ side. The existence of $\mu_{+}$ and $\mu_{-}$ might not be obvious but can be shown as follows. Take a cover $\tilde{X} \to X$ of finite index so that a lift $\tilde{\gamma}_0$ of $\gamma_0$ is simple. In $\tilde{X}$, it suffices to find orthogonal arcs that return to the same side by completing $\tilde{\gamma}_0$ into a pants decomposition and by taking appropriate orthogonal arcs in the pants containing $\tilde{\gamma}_0$. The image in $X$ of these arcs provides $\mu_{+}$ and $\mu_{-}$.)

We note that $\theta_0$, $D$, $\ell(\gamma_0)$, $\ell(\mu_{+})$ and $\ell(\mu_{-})$ are 
all quantities that depend only $X$. 

We'll begin by constructing a closed piecewise geodesic segment containing all of the $c_i$s with endpoints lying on $\gamma_0$. We begin by extending each $c_i$ to a segment $\tilde{c}_i$ with endpoints on $\gamma_0$ as follows. Following Lemma \ref{lem:tech} we set
$$
r_\varepsilon:= \log\left(\frac{1}{\varepsilon}\right) + \log\left(\frac{4}{\sin(\theta_0)}\right)
$$

We extend each $c_i$ by $r_\varepsilon$ in both directions. We continue to extend it until it intersects $\gamma_0$ at an angle $\theta\geq \theta_0$. As $\gamma_0$ fills, a first intersection point will occur within an extension of length at most $D$ on both ends. If the angle is too small in one of the endpoints, by extending by at most an additional $D$ there is a second intersection point. The arc between two successive intersection points is an arc with endpoints on one of the polygons of $X\setminus \gamma_0$. By Lemma \ref{lem:angle}, one of the two angles must be at least $\theta_0$. 

We'll also need to apply Lemma \ref{lem:angle} to the extended segments so if their current length is not yet $m(\theta_0)$ (where $m(\theta_0)$ is from Lemma \ref{lem:angle}), we extend it equally in both directions until it reaches that length. We denote the resulting segment $c'_i$. 

Note that because the quantities $D$, $\theta_0$ and $m(\theta_0)$ only depend on $X$ we have that the resulting geodesic arc $c'_i$ satisfies
$$
\ell(c'_i) \leq \ell(c_i) + 2\log\left(\frac{1}{\varepsilon}\right) + C_X
$$
where $C_X$ is a constant only depending on $X$. 

We think of the segments $c'_i$ as being cyclically ordered and we construct a closed curve at follows. 

We orient each $c_i$ arbitrarily. What will be important is on which sides of $\gamma_0$ they start and end on. 

To determine what happens between the endpoint of $c_i$ and the starting point of $c_{i+1}$ what will be important is on which sides of $\gamma_0$ they start and end on.

If $c'_{i+1}$ begins on the opposite side that $c'_i$ ends on, we join the endpoint of $c'_{i}$ to the starting point of $c'_{i+2}$ by the shortest subarc of $\gamma_0$ which does this and which will be of length at most $\frac{\gamma_0}{2}$. 

If $c'_{i+1}$ begins on the same side that $c'_i$ ends on, we'll use either $\mu_{+}$ or $\mu_{-}$ which we also think of as oriented. Suppose $c'_{i}$ ends on $+$ (the other case is symmetric). We join the endpoint of $c'_i$ to the starting point of $\mu_{-}$ by a shortest arc of $\gamma_0$ that does this and is of length at most $\frac{\gamma_0}{2}$. Now we join the endpoint of $\mu_{-}$ to the initial point of $c'_{i+1}$ similarly, i.e., by a shortest arc of $\gamma_0$ of length at most $\frac{\gamma_0}{2}$. All in all, the length of these three additional arcs is at most $\ell(\gamma_0) + \max( \ell(\mu_{+}), \ell(\mu_{-}))$. 

In this way, we have constructed a piecewise geodesic closed curve $\gamma'$ and we denote by $\gamma$ the unique geodesic in its homotopy class. Its length is thus upper bounded by
$$
\sum_{k=1}^N \ell(c'_k) + N\left(\ell(\gamma_0) + \max( \ell(\mu_{+}), \ell(\mu_{-}))\right)
$$
Putting this bound together with the bound on the length of each $c'_i$, we set
$$
K_X:= C_X + \ell(\gamma_0) + \max( \ell(\mu_{+}), \ell(\mu_{-}))
$$
to obtain the desired bound on the length of $\gamma$.

What remains to be seen is that $\gamma$ is non-trivial and contains all of the segments $c_i$ in its $\varepsilon$-neighborhood. To see this we consider a lift in the universal cover. 

Note that the piecewise geodesic $\gamma'$ consists of an even number of geodesic arcs where arcs of $\gamma_0$ are followed by either $c'_i$s or $\mu_{+}$ or $\mu_{-}$ (we'll refer to these as the non $\gamma_0$ arcs). 

We begin by taking a specific lift of $\gamma'$ (for instance by taking a lift of $c'_1$ and then constructing the full lift from there). Note that each lift of a $c'_i$ lies between two copies of $\gamma_0$ which are disjoint by the length condition imposed on $c'_i$ and by Lemma \ref{lem:angle}. As they are orthogonal in their endpoints, this also holds for lifts of $\mu_{+}$ and $\mu_{-}$. 

We think of the lift of $\gamma'$ as being oriented. We were careful about what sides the non $\gamma_0$ arcs left from and in particular we ensured that successive non $\gamma_0$ arcs left from different sides. In the lift, this means that the piecewise geodesic never backtracks and traverses each lift of $\gamma_0$ corresponding to the lift of a $\gamma_0$ arc. Thus the successive lifts of $\gamma_0$ form a sequence of nested geodesics.

\begin{figure}[h]
{\color{linkred}
\leavevmode \SetLabels
\endSetLabels
\begin{center}
\AffixLabels{\centerline{\epsfig{file =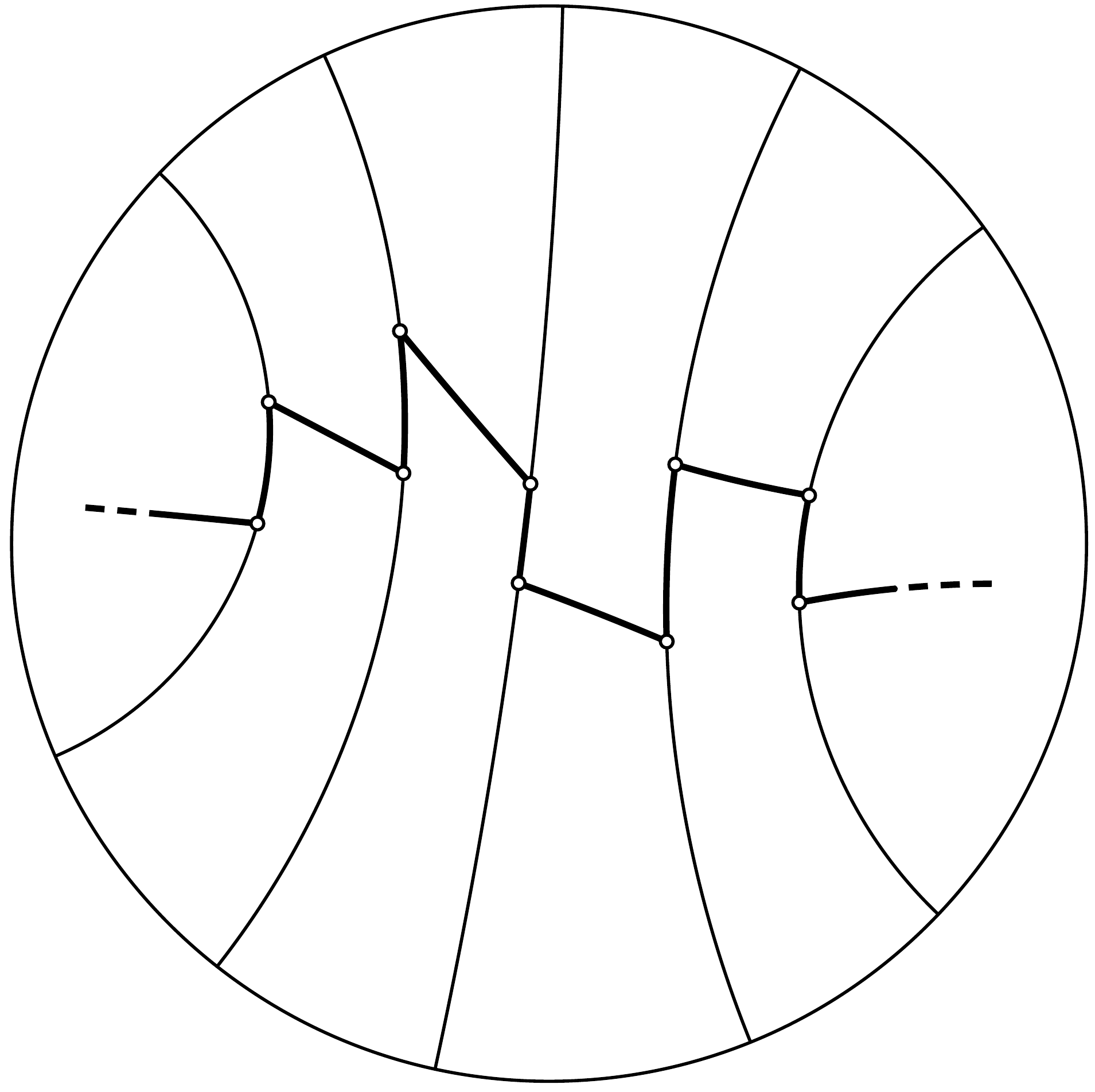,width=7.0cm,angle=0} }}
\vspace{-30pt}
\end{center}
\caption{Lifting $\gamma'$ to $\Hyp$: the complete geodesics are lifts of $\gamma_0$} \label{fig:lifts}
}
\end{figure}

That tells us two things: the lift of $\gamma'$ is simple and it limits in both directions to distinct endpoints on the boundary of $\Hyp$. We denote by $\tilde{\gamma}$ the unique geodesic in $\Hyp$ with these endpoints (this is a lift of $\gamma$). By standard arguments, this ensures that $\gamma'$, and thus $\gamma$, is non trivial.

We now look at $\tilde{\gamma}$. Given a lift of $c'_i$, we look at the two lifts of $\gamma_0$ surrounding it which we'll denote $\gamma'_i$ and $\gamma''_i$. By construction $c'_i$ has the proper length on both ends of the lift of $c_i$ and angle condition to ensure that any geodesic intersecting both $\gamma'_i$ and $\gamma''_i$ contains $c_i$ in a $\varepsilon$-neighborhood. By the above argument, $\tilde{\gamma}$ crosses both $\gamma'_i$ and $\gamma''_i$ and this is what we wanted to show. 
\end{proof}

\section{Quantitative density}

We now apply Theorem \ref{thm:maintech} to prove results about dense geodesics. Recall that a closed geodesic $\gamma$ is $\varepsilon$-dense on $X$ if $d_X(x,\gamma) \leq \varepsilon$ for all $x\in X$. 

\begin{reptheorem}{thm:denseX}
For all $X\in \M_g$ there exists a constant $C_X>0$ such that for all $\varepsilon \leq \frac{1}{2}$ there exists a closed geodesic $\gamma_\varepsilon$ that is $\varepsilon$-dense on $X$ and such that
$$
\ell(\gamma_\varepsilon) \leq C_X \frac{1}{\varepsilon}  \log\left(\frac{1}{\varepsilon}\right)
$$
\end{reptheorem}

\begin{proof}
The strategy is to find a collection of geodesic segments that $\frac{\varepsilon}{2}$-fill $X$ and then to apply Theorem \ref{thm:maintech} to find a closed geodesic that contains the segments in a $\frac{\varepsilon}{2}$-neighborhood. 

We begin by setting
$$R_X : = \min \{1, \injrad(X)\}$$
where $\injrad(X)$ is the minimum injectivity radius of $X$. We consider a maximal collection points $\{p_k\}_{k\in I}$ all pairwise distance at least $R_X$ ($I$ is a finite index set). Note that the balls of radius $R_X$ around the $p_k$ cover $X$. Further note that the balls of radius $ \frac{R_X}{2}$ around the $p_k$ are disjoint and thus 
$$
| I | \leq \frac{\area (X)}{\area\left(B_{\sfrac{R_X}{2}}\right)}
$$
In particular, there is a bound on $| I |$ which only depends on $X$. 
Now consider a ball of radius $R_X$ in $\Hyp$. As $R_X \leq 1$, the perimeter of this ball is less than
$$
2\pi \sinh(1)
$$
We want to fill it by geodesic segments such that all points are distance at most $\frac{\varepsilon}{2}$ from one of the segments. To do so we place points $q_k, k\in J$ on the perimeter, all exactly $\varepsilon$-apart except for possibly the last one which might be closer to its two neighbors. Note there are at most
$$
\frac{2\pi \sinh(1)}{\varepsilon}
$$
Now consider any collection of disjoint "parallel" segments constructed by first taking two neighbors and joining them by a geodesic segment, and then taking the geodesic between their other neighbors and so forth. There are half as many segments as there were points. Taking into account that we might have had an odd number of points, we've constructed at most 
$$
\frac{\pi \sinh(1)}{\varepsilon} + 1
$$
segments. Now any point in the disk lies between two geodesic segments and is at a distance at most $\frac{\varepsilon}{2}$ from one of them (this is because the distance between the endpoints of these segments is at most $\varepsilon$). 

Doing this for each of our balls, we have at most
$$
A_X \frac{1}{\varepsilon}
$$
segments where $A_X$ is a constant depending only on $X$. Each of them are length at most $2$. 

We now apply Theorem \ref{thm:maintech} to these segments asking that the closed geodesic $\gamma_\varepsilon$ contain every segment in its $\frac{\varepsilon}{2}$ neighborhood. The estimate follows.
\end{proof}

Using the same method as above, one can find an estimate on the length of a shortest closed geodesic that is $\varepsilon$-dense in the unit tangent bundle $UT(X)$ of $X$. For this one needs a notion of distance. One way of defining this is by associating oriented segments to vectors. Given $X$ and given a constant $l \leq \injrad(X)$, we say that a curve $\gamma$ is $\varepsilon$-dense in $UT(X)$ if every oriented geodesic segment on $X$ of length $l$ is contained in a $\varepsilon$-neighborhood of $\gamma$. 

\begin{theorem}\label{thm:tangent}
For all $X\in \M_g$ there exists a constant $C_X>0$ and an $\varepsilon_0$ such that for all $\varepsilon \leq \frac{1}{2}$ there exists a closed geodesic $\gamma_\varepsilon$ that is $\varepsilon$-dense in $UT(X)$ and such that
$$
\ell(\gamma_\varepsilon) \leq C_X \frac{1}{\varepsilon^2}  \log\left(\frac{1}{\varepsilon}\right)
$$
\end{theorem}

\begin{proof}
The proof is identical to the proof of Theorem \ref{thm:denseX} with the following exceptions.

When choosing the set of balls that cover $X$, we choose the balls to be of radius $2 R_X$ (they may not be embedded but we can consider them in the universal cover and project to $X$). This is to ensure that we capture the totality of oriented segments of length $\ell$. 

Now when constructing the geodesic segments, instead of choosing roughly $\frac{1}{\varepsilon}$ segments that lie between the points $q_k$, we choose all oriented geodesic segments that lie between them. As there are roughly $\frac{1}{\varepsilon}$ points $q_k$, there are roughly $\frac{1}{\varepsilon^2}$ segments this time, hence the difference in the estimate.
\end{proof}

We end the paper with the corollary of Theorem \ref{thm:denseX}, mentioned in the introduction.

\begin{repcorollary}{cor:int}
For all $X\in \M_g$ there exists a constant $C_X>0$ and an $\varepsilon_0$ such that for all $\varepsilon \leq \varepsilon_0$ there exists a closed geodesic $\gamma_\varepsilon$ that is $\varepsilon$-dense on $X$ and such that
$$
i(\gamma_\varepsilon, \gamma_\varepsilon) \leq C_X \frac{1}{\varepsilon^2}  \left(\log\left(\frac{1}{\varepsilon}\right)\right)^2
$$
\end{repcorollary}

\begin{proof}
This follows directly from the length estimates on $\gamma_\varepsilon$ from Theorem \ref{thm:denseX} as any geodesic of length $\ell$ can have at most roughly $\ell^2$ self-intersection points by the following argument. Take the geodesic of length $\ell$ and break it up into segments of length $\ell_0$ where $\ell_0$ is less than the minimal injectivity radius of $X$. Now any two segments of length $\ell_0$ can intersect at most once.
\end{proof}

\addcontentsline{toc}{section}{References}
\bibliographystyle{plain}

\end{document}